\documentclass[english]{article}
\usepackage[T1]{fontenc}
\usepackage[latin9]{inputenc}
\usepackage{geometry}
\usepackage{babel}

\usepackage{textcomp}
\usepackage{amsthm}
\usepackage{relsize}
\usepackage{amsmath}
\usepackage{amssymb}
\usepackage[all]{xy}
\usepackage[unicode=true, pdfusetitle,
 bookmarks=true,bookmarksnumbered=false,bookmarksopen=false,
 breaklinks=false,pdfborder={0 0 1},backref=false,colorlinks=false]
 {hyperref}

\makeatletter

\DeclareRobustCommand{\lyxmathsym}[1]{\ifmmode\begingroup\def\b@ld{bold}
  \def\rmorbf##1{\ifx\math@version\b@ld\textbf{##1}\else\textrm{##1}\fi}
  \mathchoice{\hbox{\rmorbf{#1}}}{\hbox{\rmorbf{#1}}}
  {\hbox{\smaller[2]\rmorbf{#1}}}{\hbox{\smaller[3]\rmorbf{#1}}}
  \endgroup\else#1\fi}

\theoremstyle{plain}
\newtheorem{thm}{Theorem}[section]
  \theoremstyle{plain}
  \newtheorem{conjecture}[thm]{Conjecture}
  \theoremstyle{definition}
  \newtheorem{defn}[thm]{Definition}
  \theoremstyle{plain}
  \newtheorem{prop}[thm]{Proposition}
  \theoremstyle{plain}
  \newtheorem{cor}[thm]{Corollary}
  \theoremstyle{plain}
  \newtheorem{lem}[thm]{Lemma}

\usepackage{xy}
\usepackage{mathrsfs}
\usepackage{hyperref}
\usepackage{color}
\hypersetup{colorlinks=true,citecolor=blue}

\makeatother

\begin{document}

\title{Artin vanishing in rigid analytic geometry}

\author{David Hansen%
\thanks{Department of Mathematics, Columbia University, 2990 Broadway, New
York NY 10027; hansen@math.columbia.edu%
}}
\maketitle
\begin{abstract}
We prove a rigid analytic analogue of the Artin vanishing theorem.
Precisely, we prove (under mild hypotheses) that the geometric étale
cohomology of any Zariski-constructible sheaf on any affinoid rigid
space $X$ vanishes in all degrees above the dimension of $X$. Along
the way, we show that branched covers of normal rigid spaces can often
be extended across closed analytic subsets, in analogy with a classical
result for complex analytic spaces. We also prove a general comparison
theorem relating the algebraic and analytic étale cohomologies of
any affinoid rigid space.

\tableofcontents{}
\end{abstract}

\section{Introduction}

Let $X\subset\mathbf{C}^{m}$ be a smooth affine variety over $\mathbf{C}$,
or more generally any complex Stein manifold. According to a classical
theorem of Andreotti and Frankel \cite{AF}, $X$ has the homotopy
type of a CW complex of real dimension $\leq\mathrm{dim}X$. In particular,
the cohomology groups $H^{i}(X,A)$ vanish for any abelian group $A$
and any $i>\mathrm{dim}X$. This vanishing theorem was significantly
generalized by Artin, who proved the following striking result.
\begin{thm}[Artin, Corollaire XIV.3.2 in \cite{SGA4v3}]
Let $X$ be an affine variety over a separably closed field $k$,
and let $\mathscr{F}$ be any torsion abelian sheaf on the étale site
of $X$. Then \[
H_{\mathrm{\acute{e}t}}^{i}(X,\mathscr{F})=0\]
for all $i>\mathrm{dim}X$.
\end{thm}
We remind the reader that for a general $k$-variety $X$, the groups
$H_{\mathrm{\acute{e}t}}^{i}(X,\mathscr{F})$ vanish in degrees $i>2\,\mathrm{dim}X$,
and this bound is sharp.

It's natural to wonder whether there is a rigid analytic analogue
of the Artin vanishing theorem. Again, we have a general sharp vanishing
theorem due to Berkovich and Huber (cf. \cite[Corollary 4.2.6]{BerIHES},
\cite[Corollary 2.8.3]{Hub96}): for any quasicompact and quasiseparated%
\footnote{(More generally, one can allow any quasiseparated rigid space admitting
a covering by countably many quasicompact open subsets.)%
} rigid space $X$ over a complete algebraically closed nonarchimedean
field $C$, and any torsion abelian sheaf $\mathscr{F}$ on $X_{\mathrm{\acute{e}t}}$,
the cohomology group $H_{\mathrm{\acute{e}t}}^{i}(X,\mathscr{F})$
vanishes for all $i>2\,\mathrm{dim}X$. Now in rigid geometry the
\emph{affinoid} \emph{spaces} play the role of basic affine objects,
and the most naive guess for an analogue of Artin vanishing would
be that $H_{\mathrm{\acute{e}t}}^{i}(X,\mathscr{F})$ vanishes for
all affinoids $X/C$, all torsion abelian sheaves $\mathscr{F}$ on
$X_{\mathrm{\acute{e}t}}$ and all $i>\mathrm{dim}X$. Unfortunately,
after some experimentation, one discovers that this fails miserably:
there are plenty of torsion abelian sheaves on the étale site of any
$d$-dimensional affinoid with nonzero cohomology in all\emph{ }degrees
$i\in[0,2d]$. However, the following conjecture seems to be a reasonable
salvage.
\begin{conjecture}
\label{rigidartinvanishing}Let $X$ be an affinoid rigid space over
a complete algebraically closed nonarchimedean field $C$, and let
$\mathscr{G}$ be any \emph{Zariski-constructible }sheaf of $\mathbf{Z}/n\mathbf{Z}$-modules
on $X_{\mathrm{\acute{e}t}}$ for some $n$ prime to the residue characteristic
of $C$. Then\[
H_{\mathrm{\acute{e}t}}^{i}(X,\mathscr{G})=0\]
for all $i>\mathrm{dim}X$.
\end{conjecture}
Here for a given rigid space $X$ and Noetherian coefficient ring
$\Lambda$, we say a sheaf $\mathscr{G}$ of $\Lambda$-modules on
$X_{\mathrm{\acute{e}t}}$ is Zariski-constructible if $X$ admits
a locally finite stratification into subspaces $Z_{i}\subset X$,
each locally closed for the Zariski topology on $X$, such that $\mathscr{G}|_{Z_{i,\mathrm{\acute{e}t}}}$
is a locally constant sheaf of $\Lambda$-modules of finite type for
each $i$. Note that Zariski-constructible sheaves are overconvergent,
and so it is immaterial whether one interprets their étale cohomology
in the framework of Berkovich spaces or adic spaces (cf. \cite[Theorem 8.3.5]{Hub96}).

The main result of this paper confirms Conjecture \ref{rigidartinvanishing}
in the case where $C$ has characteristic zero and the pair $(X,\mathscr{G})$
arises via base extension from a discretely valued nonarchimedean
field.
\begin{thm}
\label{avfull}Let $X$ be an affinoid rigid space over a complete
discretely valued nonarchimedean field $K$ of characteristic zero,
and let $\mathscr{F}$ be any Zariski-constructible sheaf of $\mathbf{Z}/n\mathbf{Z}$-modules
on $X_{\mathrm{\acute{e}t}}$ for some $n$ prime to the residue characteristic
of $K$. Then the cohomology groups $H_{\mathrm{\acute{e}t}}^{i}(X_{\widehat{\overline{K}}},\mathscr{F})$
are finite for all $i$, and \[
H_{\mathrm{\acute{e}t}}^{i}(X_{\widehat{\overline{K}}},\mathscr{F})=0\]
for all $i>\mathrm{dim}X$.
\end{thm}
For a slightly more general result, see Corollary \ref{steincase}.
As far as we know, this is the first progress on Conjecture \ref{rigidartinvanishing}
since Berkovich \cite{BerVanishing2} treated some cases where $\mathscr{F}=\mathbf{Z}/n\mathbf{Z}$
is constant and $X$ is assumed algebraizable in a certain sense.
In particular, using a deep algebraization theorem of Elkik \cite[Th\'eor\`eme 7]{Elkik},
Berkovich proved Conjecture \ref{rigidartinvanishing} when $\mathscr{F}$
is constant and $X$ is \emph{smooth, }which might give one some confidence
in the general conjecture.

Our proof of Theorem \ref{avfull} doesn't explicitly use any algebraization
techniques. Instead, we reduce to the special case where $\mathscr{F}$
is constant. In this situation, it turns out we can argue directly,
with fewer assumptions on $K$:
\begin{thm}
\label{avconstant}Let $X=\mathrm{Spa}\, A$ be an affinoid rigid
space over a complete discretely valued nonarchimedean field $K$.
Then\[
H_{\mathrm{\acute{e}t}}^{i}(X_{\widehat{\overline{K}}},\mathbf{Z}/n\mathbf{Z})=0\]
for all $i>\mathrm{dim}X$ and all $n$ prime to the residue characteristic
of $K$.
\end{thm}
The proof of this theorem uses a number of ingredients, including
some theorems of Greco and Valabrega on excellent rings, a remarkable
formula of Huber for the stalks of the nearby cycle sheaves $R^{q}\lambda_{\ast}(\mathbf{Z}/n\mathbf{Z})$
on $\mathrm{Spec}(A^{\circ}/\varpi)_{\mathrm{et}}$, a special case
of Gabber's delicate {}``affine Lefschetz theorem'' for quasi-excellent
schemes, and the classical Artin vanishing theorem.

The reduction step involves an ingredient which seems interesting
in its own right. To explain this, we make the following definition.
\begin{defn}
Let $X$ be a normal rigid space. A \emph{cover }of $X$ is a finite
surjective map $\pi:Y\to X$ from a normal rigid space $Y$, such
that there exists some closed nowhere-dense analytic subset $Z\subset X$
with $\pi^{-1}(Z)$ nowhere-dense and such that $Y\smallsetminus\pi^{-1}(Z)\to X\smallsetminus Z$
is finite étale.
\end{defn}
We then have the following result, which seems to be new.
\begin{thm}
\label{extendingcovers}Let $X$ be a normal rigid space over a complete
nonarchimedean field $K$, and let $Z\subset X$ be any closed nowhere-dense
analytic subset. Then the restriction functor\begin{eqnarray*}
\left\{ \begin{array}{c}
\mathrm{covers\, of\,}X\\
\mathrm{\acute{e}tale\, over\,}X\smallsetminus Z\end{array}\right\}  & \to & \left\{ \begin{array}{c}
\mathrm{finite\,\acute{e}tale\, covers}\\
\mathrm{of\,}X\smallsetminus Z\end{array}\right\} \\
Y & \mapsto & Y\times_{X}(X\smallsetminus Z)\end{eqnarray*}
is fully faithful. Moreover, if $K$ has characteristic zero, it is
an equivalence of categories; in other words, any finite étale cover
of $X\smallsetminus Z$ extends uniquely to a cover of $X$.
\end{thm}
We remind the reader that in the schemes setting, the analogue of
the equivalence in Theorem \ref{extendingcovers} is an easy exercise
in taking normalizations, and holds essentially whenever the base
scheme $X$ is Nagata, while for complex analytic spaces the problem
was solved by Stein and Grauert-Remmert in the 50's, cf. \cite{DethloffGrauert}.

Let us say something about the proof. Full faithfulness is an easy
consequence of Bartenwerfer's rigid analytic version of Riemann's
Hebbarkeitssatz, which says that bounded functions on normal rigid
spaces extend uniquely across nowhere-dense closed analytic subsets.
Essential surjectivity in characteristic zero is more subtle; indeed,
it provably fails in positive characteristic. When $X$ is smooth
and $Z$ is a strict normal crossings divisor, however, essential
surjectivity was proved by Lütkebohmert in his work \cite{Lutkebohmert}
on Riemann's existence problem.%
\footnote{Although curiously, Lütkebohmert doesn't explicitly state the result
in his paper, nor does he discuss full faithfulness.%
} We reduce the general case to Lütkebohmert's result using recent
work of Temkin on embedded resolution of singularities for quasi-excellent
schemes in characteristic zero

As another application of this circle of ideas, we prove the following
very general comparison result.
\begin{thm}
\label{comparison}Let $X=\mathrm{Spa}\, A$ be any affinoid rigid
space, and set $\mathcal{X}=\mathrm{Spec}\, A$, so there is a natural
morphism of sites $\mu_{X}:X_{\mathrm{\acute{e}t}}\to\mathcal{X}_{\mathrm{\acute{e}t}}$.
Then for any torsion abelian sheaf $\mathscr{F}$ on $\mathcal{X}_{\mathrm{\acute{e}t}}$,
the natural comparison map\[
H_{\mathrm{\acute{e}t}}^{n}(\mathcal{X},\mathscr{F})\to H_{\mathrm{\acute{e}t}}^{n}(X,\mu_{X}^{\ast}\mathscr{F})\]
is an isomorphism for all $n\geq0$.
\end{thm}
When $\mathscr{F}$ is a \emph{constant} sheaf, this is proved in
Huber's book; we reduce the general case of Theorem \ref{comparison} to Huber's result, using the schemes version of Theorem \ref{extendingcovers} together
with a cohomological descent argument. In particular, we use the fact
that any surjective qcqs morphism of rigid analytic spaces is universally
of cohomological descent relative to the \'etale topology, which seems
to be a new observation.

We end the introduction with a conjecture, which would give some further
justification for the definition of Zariski-constructible sheaves.
\begin{conjecture}
In the notation of the previous theorem, the pullback functor \[
\mu_{X}^{\ast}:\mathrm{Sh}(\mathcal{X}_{\mathrm{\acute{e}t}},\Lambda)\to\mathrm{Sh}(X_{\mathrm{\acute{e}t}},\Lambda)\]
induces an equivalence of categories from constructible étale sheaves
on $\mathcal{X}$ to Zariski-constructible étale sheaves on $X$.
\end{conjecture}
With Theorem \ref{extendingcovers} in hand, one can reduce this conjecture
in the characteristic zero case to a comparison of sheaf Exts on $\mathcal{X}_{\mathrm{\acute{e}t}}$
and $X_{\mathrm{\acute{e}t}}$. The essential point in this reduction
is that while open subsets $\mathcal{U}\subset\mathcal{X}$ are in
tautological bijection with Zariski-open subsets $U\subset X$, it's
\emph{a priori} difficult to understand the essential image of the
analytification functor $\mathcal{U}_{\mathrm{f\acute{e}t}}\mapsto U_{\mathrm{f\acute{e}t}}$
for a given $U$, since $ $$U$ is essentially never affinoid; however,
as an easy consequence of Theorem \ref{extendingcovers}, one gets
an equivalence of categories $\mathcal{U}_{\mathrm{f\acute{e}t}}\cong U_{\mathrm{f\acute{e}t}}$
(at least when $X$ is normal). Using this, it's easy to deduce that
if $\mathscr{F}$ is a Zariski-constructible sheaf on $X$ and $X=\coprod Z_{i}$
is a suitably nice partition such that $\mathscr{F}|_{Z_{i}}$ is
locally constant, than each $\mathscr{F}|_{Z_{i}}$ is the pullback
of a locally constant sheaf on the corresponding subset $\mathcal{Z}_{i}\subset\mathcal{X}$. 

Finally, let us note that our argument for the reduction of Theorem
\ref{avfull} to Theorem \ref{avconstant} reduces Conjecture \ref{rigidartinvanishing}
to the special case where $\mathscr{F}$ is constant, at least for
$C$ of characteristic zero. With Theorem \ref{avconstant} in hand,
this might put the general characteristic zero case of Conjecture
\ref{rigidartinvanishing} within reach of some approximation argument.

\subsubsection*{Remarks on terminology and conventions. }

Our convention is that a {}``nonarchimedean field'' is a topological
field whose topology is defined by a \emph{nontrivial }nonarchimedean
valuation of rank one. If $K$ is any nonarchimedean field, we regard
rigid analytic spaces over $K$ as a full subcategory of the category
of adic spaces locally of topologically finite type over $\mathrm{Spa}(K,K^{\circ})$.
If $A$ is any topological ring, we write $A^{\circ}$ for the subset
of power-bounded elements; if $A$ is a Huber ring, we write $\mathrm{Spa\,}A$
for $\mathrm{Spa}(A,A^{\circ})$. 

We use the terms {}``Zariski-closed subset'' and {}``closed analytic
subset'' interchangeably, and we always regard Zariski-closed subsets
of rigid spaces as rigid spaces via the induced reduced structure.
Finally, we remind the reader that in rigid geometry the phrases {}``dense
Zariski-open subset'' and {}``Zariski-dense open subset'' have
very different meanings.

\subsection*{Acknowledgments}

Several years ago, Giovanni Rosso and John Welliaveetil asked me whether
anything was known about the $l$-cohomological dimension of affinoid
rigid spaces, and I'd like to thank them very heartily for this initial
stimulation. 

Johan de Jong listened to some of my early ideas about these problems
and pointed me to the {}``Travaux de Gabber'' volumes in response
to my desperate search for tools. Christian Johansson planted in my
head the usefulness of the excellence property for convergent power
series rings. Peter Scholze read a rough initial draft and pointed
out that certain arguments worked without change in equal characteristic
zero; Scholze also inadvertently inspired a crucial trick (cf. footnote
7). Brian Conrad offered several helpful comments and corrections
on a later draft. To all of these mathematicians, I'm very grateful.

\section{Preliminaries}

\subsection{Zariski-constructible sheaves on rigid spaces}

In this section we discuss some basics on Zariski-constructible étale
sheaves on rigid spaces. For simplicity we fix a nonachimedean field
$K$ and a Noetherian coefficient ring $\Lambda$; until further notice,
roman letters $X,Y,...$ denote rigid spaces over $K$, and {}``Zariski-constructible''
means a Zariski-constructible sheaf of $\Lambda$-modules, as defined
in the introduction, on the étale site of some rigid space $X$ over
$K$. We'll also say that an object $\mathscr{F}\in D^{b}(X_{\mathrm{\acute{e}t}},\Lambda)$
is Zariski-constructible if it has Zariski-constructible cohomology
sheaves.
\begin{prop}
Let $\mathscr{F}$ be a Zariski-constructible sheaf on a rigid space
$X$.

\emph{i. }If $f:Y\to X$ is any morphism of rigid spaces, then $f^{\ast}\mathscr{F}$
is Zariski-constructible.

\emph{ii. }If $i:X\to W$ is a closed immersion, then $i_{\ast}\mathscr{F}$
is Zariski-constructible.

\emph{iii. }If $j:X\to V$ is a Zariski-open immersion and $\mathscr{F}$
is locally constant, then $j_{!}\mathscr{F}$ is Zariski-constructible.\end{prop}
\begin{proof}
Trivial.
\end{proof}
We will often verify Zariski-constructibility via the following \emph{dévissage},
which is a trivial consequence of the previous proposition.
\begin{prop}
Let $X$ be any rigid space, and let $\mathscr{F}$ be a sheaf of
$\Lambda$-modules on $X_{\mathrm{\acute{e}t}}$. The following are
equivalent:

\emph{i. }$\mathscr{F}$ is Zariski-constructible.

\emph{ii. }There is some dense Zariski-open subset $j:U\to X$ with
closed complement $i:Z\to X$ such that $i^{\ast}\mathscr{F}$ is
Zariski-constructible and $j^{\ast}\mathscr{F}$ is locally constant
of finite type.
\end{prop}
Note that one probably cannot weaken the hypotheses in ii. here to
the condition that $j^{\ast}\mathscr{F}$ is Zariski-constructible;
this is related to the fact that the Zariski topology in the rigid
analytic world is not transitive.
\begin{prop}
\label{finitepush}If $f:X'\to X$ is a finite morphism and $\mathscr{F}$
is a Zariski-constructible sheaf on $X'$, then $f_{\ast}\mathscr{F}$
is Zariski-constructible.
\end{prop}
We note in passing that if $f:X\to Y$ is a finite morphism, or more
generally any quasi-compact separated morphism with finite fibers,
then $f_{\ast}:\mathrm{Sh}(X_{\mathrm{\acute{e}t}},\Lambda)\to\mathrm{Sh}(Y_{\mathrm{\acute{e}t}},\Lambda)$
is an exact functor, cf. Proposition 2.6.4 and Lemma 1.5.2 in \cite{Hub96}.
\begin{proof}
We treat the case where $X'=\mathrm{Spa}\, A'$ and $X=\mathrm{Spa}\, A$
are affinoid, which is all we'll need later. We can assume they are
reduced and that $f$ is surjective. If $i:Z\to X$ is Zariski-closed
and nowhere dense, then $\mathrm{dim}Z<\mathrm{dim}X$; setting $Z'=Z\times_{X}X'$
and writing $f':Z'\to Z$ and $i':Z'\to X'$ for the evident morphisms,
we can assume that $i^{\ast}f_{\ast}\mathscr{F}\cong f'_{\ast}i'^{\ast}\mathscr{F}$
is Zariski-constructible by induction on $\mathrm{dim}X$. By dévissage,
it now suffices to find a dense Zariski-open subset $j:U\to X$ such
that $j^{\ast}f_{\ast}\mathscr{F}$ is locally constant. To do this,
choose a dense Zariski-open subset $V\subset X'$ such that $\mathscr{F}|_{V}$
is locally constant. Then $W=X\smallsetminus f(X'\smallsetminus V)$
is a dense Zariski-open subset of $X$, and $\mathscr{F}$ is locally
constant after pullback along the open immersion $W'=W\times_{X}X'\to X'$.
If $\mathrm{char}(K)=0$, we now conclude by taking $U$ to be any
dense Zariski-open subset contained in $W$ such that $U'=U\times_{X}X'\to U$
is finite étale; if $\mathrm{char}(K)=p$, we instead choose $U$
so that $U'\to U$ factors as the composition of a universal homeomorphism
followed by a finite étale map. (For the existence of such a $U$,
look at the map of schemes $\mathrm{Spec}\, A'\to\mathrm{Spec}\, A$;
this morphism has the desired structure over all generic points of
the target, and these structures then spread out over a dense Zariski-open
subset of $\mathrm{Spec}\, A$. One then concludes by analytifying.)
\end{proof}
Next we check the two-out-of-three property:
\begin{prop}
\label{twooutofthree}Let $X$ is a rigid space and let $0\to\mathscr{F}\to\mathscr{G}\to\mathscr{H}\to0$
be a short exact sequence of étale sheaves of $\Lambda$-modules on
$X$. If two of the three sheaves $\{\mathscr{F},\mathscr{G},\mathscr{H}\}$
are Zariski-constructible, then so is the third.
\end{prop}
Using this result, it's straightforward to check that Zariski-constructible
sheaves form an abelian subcategory of $\mathrm{Sh}(X_{\mathrm{\acute{e}t}},\Lambda)$;
since we never need this result, we leave it as an exercise for the
interested reader.
\begin{proof}
By induction on $\mathrm{dim}X$ and dévissage, it suffices to find
some dense Zariski-open subset $j:U\to X$ such that all three sheaves
are locally constant after restriction to $U$. By assumption, we
can choose $U$ such that two of the three sheaves have this property.
Looking at the exact sequence $0\to\mathscr{F}|_{U}\to\mathscr{G}|_{U}\to\mathscr{H}|_{U}\to0$,
\cite[Lemma 2.7.3]{Hub96} implies that all three sheaves are constructible
(in the sense of \cite{Hub96}). By \cite[Lemma 2.7.11]{Hub96}, it
now suffices to check that all three sheaves are overconvergent. But
if $\overline{x}\rightsquigarrow\overline{y}$ is any specialization
of geometric points, this follows immediately by applying the snake
lemma to the diagram\[
\xymatrix{0\ar[r] & \mathscr{F}_{\overline{y}}\ar[r]\ar[d] & \mathscr{G}_{\overline{y}}\ar[r]\ar[d] & \mathscr{H}_{\overline{y}}\ar[r]\ar[d] & 0\\
0\ar[r] & \mathscr{F}_{\overline{x}}\ar[r] & \mathscr{G}_{\overline{x}}\ar[r] & \mathscr{H}_{\overline{x}}\ar[r] & 0}
\]
since by assumption two of the three vertical arrows are isomorphisms.
\end{proof}
It seems interesting to develop some more theory around Zariski-constructible
sheaves on rigid spaces. In particular, we expect this notion to have
some non-trivial stabilities under certain of the six functors:
\begin{conjecture}
Let $X$ be a finite-dimensional rigid space over a complete algebraically
closed field $C$. For simplicity, assume that $\Lambda=\mathbf{Z}/n\mathbf{Z}$
for some $n$ prime to the residue characteristic, and suppose that
$X$ admits a dualizing complex $\omega_{X}\in D^{b}(X_{\mathrm{\acute{e}t}},\Lambda)$.%
\footnote{It seems plausible that $\omega_{X}$ can be defined for essentially
any $X$ by constructing it locally on an étale hypercover by disjoint
unions of affinoids and then applying some version of the {}``BBD
gluing lemma''; the only nonformal input one should need is that
$\mathrm{Ext}_{D^{b}(U_{\mathrm{\acute{e}t}},\Lambda)}^{i}(\omega_{U},\omega_{U})=0$
for any affinoid $U$ and any $i<0$, which is immediate from \cite[Proposition B.3.1]{Happ}.%
} Let $\mathscr{F}\in D^{b}(X_{\mathrm{\acute{e}t}},\Lambda)$ be any
Zariski-constructible object. Then:

\emph{i. }The Verdier dual $\mathbf{D}_{X}\mathscr{F}\overset{\mathrm{def}}{=}R\mathscr{H}\mathrm{om}(\mathscr{F},\omega_{X})$
is Zariski-constructible.

\emph{ii. }If $f:Y\to X$ is any morphism which is finite-dimensional
and compactifiable in the sense of \cite[Definition 5.1.1]{Hub96},
so $Rf_{!}$ and $Rf^{!}$ are defined, then $Rf^{!}\mathscr{F}$
is Zariski-constructible.

\emph{iii. }If $j:X\to W$ is a Zariski-open immersion and $\mathscr{F}$
is locally constant, then $Rj_{\ast}\mathscr{F}$ is Zariski-constructible.

\emph{iv. }If $f:X\to S$ is any proper morphism of rigid spaces,
then $Rf_{\ast}\mathscr{F}$ is Zariski-constructible. 
\end{conjecture}
It is true, but not obvious, that we have implications $\mathrm{iii.\Leftrightarrow i.\Rightarrow ii.}$
among these statements, and that $\mathrm{iv}.$ is true when $\mathrm{dim}S=1$;
we omit the proofs, since these results aren't needed in the present
paper. We hope to return to this conjecture in a future article.

\subsection{Extending covers across closed subsets}

In this section we prove a slight strengthening of Theorem \ref{extendingcovers}.
We'll freely use basic facts about irreducible components of rigid
spaces, as developed in \cite{ConradIrred}, without any comment.
The following result of Bartenwerfer \cite[\S3]{Bar} is also crucial
for our purposes.
\begin{thm}[Bartenwerfer]
\label{hebbarkeitsatz}Let $X$ be a normal rigid space, and let
$Z\subset X$ be a nowhere-dense closed analytic subset, with $j:X\smallsetminus Z\to X$
the inclusion of the open complement. Then $\mathcal{O}_{X}^{+}\overset{\sim}{\to}j_{\ast}\mathcal{O}_{X\smallsetminus Z}^{+}$
and $\mathcal{O}_{X}\overset{\sim}{\to}\left(j_{\ast}\mathcal{O}_{X\smallsetminus Z}^{+}\right)[\tfrac{1}{\varpi}]$.
In particular, if $X$ is affinoid and $f\in\mathcal{O}_{X}(X\smallsetminus Z)$
is bounded, then $f$ extends uniquely to an element of $\mathcal{O}_{X}(X)$,
so $\mathcal{O}_{X}(X)\cong\mathcal{O}_{X}^{+}(X\smallsetminus Z)[\tfrac{1}{\varpi}]$.\end{thm}
\begin{cor}
If $X$ is a connected normal rigid space and $Z\subset X$ is a nowhere-dense
closed analytic subset, then $X\smallsetminus Z$ is connected.\end{cor}
\begin{proof}
Any idempotent in $\mathcal{O}_{X}(X\smallsetminus Z)$ is power-bounded,
so this is immediate from the previous theorem.\end{proof}
\begin{prop}
Let $X$ be a normal rigid space, and let $\pi:Y\to X$ be a cover
of $X$. Then each irreducible component of $Y$ maps surjectively
onto some irreducible component of $X$. Moreover, if $V\subset X$
is any closed nowhere-dense analytic subset, then $\pi^{-1}(V)$ is
nowhere-dense.\end{prop}
\begin{proof}
We immediately reduce to the case where $X$ is connected. Let $Z\subset X$
be as in the definition of a cover, and let $Y_{i}$ be any connected
component of $Y$, so then $Y_{i}\cap\pi^{-1}(Z)$ is closed and nowhere-dense
in $Y_{i}$ and $Y_{i}\smallsetminus Y_{i}\cap\pi^{-1}(Z)\to X\smallsetminus Z$
is finite étale. Then \[
\mathrm{im}\left(Y_{i}\smallsetminus Y_{i}\cap\pi^{-1}(Z)\to X\smallsetminus Z\right)\]
is a nonempty open and closed subset of $X\smallsetminus Z$, so it
coincides with $X\smallsetminus Z$ by the previous corollary. In
particular, $\pi(Y_{i})$ contains a dense subset of $X$. On the
other hand, $\pi(Y_{i})$ is a closed analytic subset of $X$ since
$\pi$ is finite. Therefore $\pi(Y_{i})=X$. 

For the second claim, note that if $V$ is a closed analytic subset
of a connected normal space $X$, then $V\subsetneq X$ if and only
if $V$ is nowhere-dense if and only if $ $$\mathrm{dim}V<\mathrm{dim}X$.
Since \[
\mathrm{dim}\pi^{-1}(V)\cap Y_{i}=\mathrm{dim}V<\mathrm{dim}X=\mathrm{dim}Y_{i}\]
for any irreducible component $Y_{i}$ of $Y$, this gives the claim.\end{proof}
\begin{prop}
Let $X$ be a normal rigid space, and let $Z\subset X$ be any closed
nowhere-dense analytic subset. Then the restriction functor \begin{eqnarray*}
\left\{ \begin{array}{c}
\mathrm{covers\, of\,}X\end{array}\right\}  & \to & \left\{ \begin{array}{c}
\mathrm{covers\, of\,}X\smallsetminus Z\end{array}\right\} \\
Y & \mapsto & Y\times_{X}(X\smallsetminus Z)\end{eqnarray*}
is fully faithful.\end{prop}
\begin{proof}
If $\pi:Y\to X$ is any cover and $U\subset X$ is any open affinoid,
then $\pi^{-1}(U)$ is affinoid as well, and $\pi^{-1}(Z\cap U)$
is nowhere-dense in $U$ by the previous proposition. But then $\mathcal{O}_{Y}(\pi^{-1}(U))\cong\mathcal{O}_{Y}^{+}(\pi^{-1}(U\smallsetminus U\cap Z))[\tfrac{1}{\varpi}]$
by Theorem \ref{hebbarkeitsatz}, so $\mathcal{O}_{Y}(U)$ only depends
on $Y\times_{X}(X\smallsetminus Z)$. This immediately gives the result.
\end{proof}
It remains to prove the following result.
\begin{thm}
\label{coversextend}Let $X$ be a normal rigid space over a characteristic
zero complete nonarchimedean field $K$, and let $Z\subset X$ be
any closed nowhere-dense analytic subset. Then the restriction functor\begin{eqnarray*}
\left\{ \begin{array}{c}
\mathrm{covers\, of\,}X\\
\mathrm{\acute{e}tale\, over\,}X\smallsetminus Z\end{array}\right\}  & \to & \left\{ \begin{array}{c}
\mathrm{finite\,\acute{e}tale\, covers}\\
\mathrm{of\,}X\smallsetminus Z\end{array}\right\} \\
Y & \mapsto & Y\times_{X}(X\smallsetminus Z)\end{eqnarray*}
is essentially surjective. 
\end{thm}
In other words, given a (surjective) finite étale cover $\pi:Y\to X\smallsetminus Z$,
we need to find a cover $\tilde{\pi}:\tilde{Y}\to X$ and an open
immersion $Y\to\tilde{Y}$ such that the diagram\[
\xymatrix{Y\ar[r]\ar[d]_{\pi} & \tilde{Y}\ar[d]^{\tilde{\pi}}\\
X\smallsetminus Z\ar[r] & X}
\]
is cartesian. We refer to this as the problem of \emph{extending }$Y$
to a cover of $X$. Note that by the full faithfulness proved above,
we're always free to work locally on $X$ when extending a given cover
of $X\smallsetminus Z$.

Until further notice, fix $K$ of characteristic zero. The key special
case is the following result.
\begin{thm}[Lütkebohmert]
\label{lutextendingsnc} If $X$ is a smooth rigid space and $D\subset X$
is a strict normal crossings divisor, then any finite étale cover
of $X\smallsetminus D$ extends to a cover of $X$.
\end{thm}
This is more or less an immediate consequence of the arguments in
§3 in Lütkebohmert's paper \cite{Lutkebohmert} (and is implicit in
the proof of Theorem 3.1 of loc. cit.). For the convenience of the
reader, we explain the deduction in detail.  Let $\mathbf{B}^{r}=\lyxmathsym{ }\mathrm{Spa\,}K\left\langle X_{1},\dots,X_{r}\right\rangle $
denotes the $r$-dimensional closed affinoid ball.
\begin{lem}[Lemma 3.3 in \cite{Lutkebohmert}]
\label{lutcruciallemma}Let $S$ be a smooth $K$-affinoid space,
and let $r\geq1$ be any integer. If $Y_{0}$ is a cover of $S\times\left(\mathbf{B}^{r}\smallsetminus V(X_{1},\dots,X_{r})\right)$
which is étale over $S\times(\mathbf{B}^{r}\smallsetminus V(X_{1}\cdots X_{r}))$,
then $ $$Y_{0}$ extends to a cover $\tilde{Y}$ of $S\times\mathbf{B}^{r}$.
\end{lem}
We also need a result of Kiehl on the existence of {}``tubular neighborhoods''
of strict normal crossings divisors in smooth rigid spaces.
\begin{lem}
\label{tubularnbhood}If $D\subset X$ is a strict normal crossings
divisor in a smooth rigid space, then for any (adic) point $x$ in
$X$ contained in exactly $r$ irreducible components $D_{1},\dots,D_{r}$
of $D$, we can find some small open affinoid $U\subset X$ containing
$x$ together with a smooth affinoid $S$ and an isomorphism $U\simeq S\times\mathbf{B}^{r}$,
under which the individual components $D_{i}\cap U$ containing $x$
identify with the zero loci of the coordinate functions $X_{i}\in\mathcal{O}(\mathbf{B}^{r})$.\end{lem}
\begin{proof}
This follows from a careful reading of Theorem 1.18 in \cite{Kiehl}
(cf. also \cite[Theorem 2.11]{Mitsui}).
\end{proof}
Granted these results, we deduce Theorem \ref{lutextendingsnc} as
follows. By full faithfulness we can assume that $X$ is quasicompact,
or even affinoid.  We now argue by induction on the maximal number
$\iota(D)$ of irreducible components of $D$ passing through any
individual point of $X$. If $\iota(D)=1$, then $D$ is smooth, so
arguing locally around any point in $D$, Lemma \ref{tubularnbhood}
puts us exactly in the situation covered by the case $r=1$ of Lemma
\ref{lutcruciallemma}. If $\iota(D)=n,$ then locally on $X$ we
can assume that $D$ has (at most) $n$ smooth components $D_{1},D_{2},\dots,D_{n}$.
By the induction hypothesis, any finite étale cover $Y$ of $X\smallsetminus D$
extends to a cover $Y_{i}$ of  $X\smallsetminus D_{i}$ for each
$1\leq i\leq n$, since $\iota(D\smallsetminus D_{i})\leq n-1$ for
$D\smallsetminus D_{i}$ viewed as a strict normal crossings divisor
in $X\smallsetminus D_{i}$.  By full faithfulness the $Y_{i}$'s
glue to a cover $Y_{0}$ of $X\smallsetminus\cap_{1\leq i\leq n}D_{i}$,
and locally around any point in $\cap_{1\leq i\leq n}D_{i}$ Lemma
\ref{tubularnbhood} again puts us in the situation handled by Lemma
\ref{lutcruciallemma}, so $Y_{0}$ extends to a cover $\tilde{Y}$
of $X$, as desired.
\begin{proof}[Proof of Theorem \ref{coversextend}]
We can assume that $X=\mathrm{Spa}\, A$ is an affinoid rigid space,
so $Z=\mathrm{Spa}\, B$ is also affinoid, and we get a corresponding
closed immersion of schemes $\mathcal{Z}=\mathrm{Spec}\, B\to\mathcal{X}=\mathrm{Spec}\, A$.
 These are quasi-excellent schemes over $\mathbf{Q}$, so according
to Theorem 1.11 in \cite{Temkin}, we can find a projective birational
morphism $f:\mathcal{X}'\to\mathcal{X}$ such that $\mathcal{X}'$
is regular and $f^{-1}(\mathcal{Z})^{\mathrm{red}}$ is a strict normal
crossings divisor, and such that $f$ is an isomorphism away from
$\mathcal{Z}\cup\mathcal{X}^{\mathrm{sing}}$.  Analytifying, we get
a proper morphism of rigid spaces $g:X'\to X$ with $X'$ smooth such
that $g^{-1}(Z)^{\mathrm{red}}$ is a strict normal crossings divisor.

Suppose now that we're given a finite étale cover $Y\to X\smallsetminus Z$.
  Base changing along $g$, we get a finite étale cover of $X'\smallsetminus g^{-1}(Z$),
which then extends to a cover $h:Y'\to X'$ by Theorem \ref{lutextendingsnc}.
Now, since $g\circ h$ is proper, the sheaf $(g\circ h)_{\ast}\mathcal{O}_{Y'}$
defines a sheaf of coherent $\mathcal{O}_{X}$-algebras by \cite{KiehlCoherence}.
Taking the normalization of the affinoid space associated with the
global sections of this sheaf, we get a normal affinoid $Y''$ together
with a finite map $Y''\to X$ and a canonical isomorphism $Y''|_{(X\smallsetminus Z)^{\mathrm{sm}}}\cong Y|_{(X\smallsetminus Z)^{\mathrm{sm}}}$.
The cover $\tilde{Y}\to X$ we seek can then be defined as the Zariski
closure of $Y''|_{(X\smallsetminus Z)^{\mathrm{sm}}}$ in $Y''$;
note that this is just a union of irreducible components of $Y''$,
so it's still normal, and it's easy to check that $\tilde{Y}$ is
a cover of $X$. Finally, since $\tilde{Y}$ and $Y$ are canonically
isomorphic after restriction to $(X\smallsetminus Z)^{\mathrm{sm}}$,
the full faithfulness argument shows that this isomorphism extends
to an isomorphism $\tilde{Y}|_{X\smallsetminus Z}\cong Y$, since
$(X\smallsetminus Z)^{\mathrm{sm}}$ is a dense Zariski-open subset
of $X\smallsetminus Z$. This concludes the proof.
\end{proof}
For completeness, we state the following mild generalization of Theorem
\ref{coversextend}.
\begin{thm}
Let $X$ be a normal rigid space over a characteristic zero complete
nonarchimedean field $K$, and let $V\subset X$ be any closed nowhere-dense
analytic subset. Suppose that $Y\to X\smallsetminus V$ is a cover,
and that there is some closed nowhere-dense analytic set $W\subset X\smallsetminus V$
such that $V\cup W$ is an analytic set in $X$ and such that \[
Y\times_{(X\smallsetminus V)}(X\smallsetminus V\cup W)\to X\smallsetminus V\cup W\]
is finite étale. Then $Y$ extends to a cover $\tilde{Y}\to X$.\end{thm}
\begin{proof}
Apply Theorem \ref{coversextend} with $Z=V\cup W$ to construct $\tilde{Y}\to X$
extending \[
Y\times_{(X\smallsetminus V)}(X\smallsetminus V\cup W)\to X\smallsetminus V\cup W,\]
and then use full faithfulness to deduce that $\tilde{Y}|_{X\smallsetminus V}\cong Y$.
\end{proof}
Combining this extension theorem with classical Zariski-Nagata purity,
we get a purity theorem for rigid analytic spaces.
\begin{cor}
Let $X$ be a smooth rigid analytic space over a characteristic zero
complete nonarchimedean field, and let $Z\subset X$ be any closed
analytic subset which is everywhere of codimension $\geq2$.  Then
finite étale covers of $X$ are equivalent to finite étale covers
of $X\smallsetminus Z$.
\end{cor}

\section{Vanishing and comparison theorems}

\subsection{The affinoid comparison theorem}

In this section we prove Theorem \ref{comparison}. Note that when
$\mathscr{F}$ is a constant sheaf of torsion abelian groups, this
theorem is exactly Corollary 3.2.3 in \cite{Hub96}, and we'll eventually
reduce to this case. The crucial input is the following lemma.
\begin{lem}
\label{keylemma}Let $A$ be a normal $K$-affinoid over a complete
nonarchimedean field $K$; set $\mathcal{X}=\mathrm{Spec}\, A$ and
$X=\mathrm{Spa}\, A$. Let $j:\mathcal{U}\to\mathcal{X}$ be the inclusion
of any Zariski-open subset, with analytification $j^{\mathrm{an}}:U\to X$,
and let $ $$\mathscr{F}$ be a locally constant constructible sheaf
of $\mathbf{Z}/m\mathbf{Z}$-modules on $\mathcal{U}_{\mathrm{\acute{e}t}}$
for some $m$. Then writing $\mu_{X}:X_{\mathrm{\acute{e}t}}\to\mathcal{X}_{\mathrm{\acute{e}t}}$
as before, the natural map\[
H_{\mathrm{\acute{e}t}}^{n}(\mathcal{X},j_{!}\mathscr{F})\to H_{\mathrm{\acute{e}t}}^{n}(X,\mu_{X}^{\ast}j_{!}\mathscr{F})\]
is an isomorphism for all $n\geq0$.
\end{lem}
In what follows we usually write $\mathscr{F}^{\mathrm{an}}=\mu_{X}^{\ast}\mathscr{F}$
when context is clear. 

Before continuing, note that if $j:\mathcal{U}\to\mathcal{X}$ is
any open immersion with closed complement $i:\mathcal{Z}\to\mathcal{X}$,
the four functors $j_{!},j^{\ast}$, $i_{\ast}$, $i^{\ast}$ and
their analytifications can be canonically and functorially commuted
with the appropriate $\mu^{\ast}$'s in the evident sense. Indeed,
for $j^{\ast}$ and $i^{\ast}$ this is obvious (by taking adjoints
of the obvious equivalences $j_{\ast}\mu_{U\ast}\cong\mu_{X\ast}j_{\ast}^{\mathrm{an}}$
and $i_{\ast}\mu_{Z\ast}\cong\mu_{X\ast}i_{\ast}^{\mathrm{an}}$),
for $j_{!}$ it is a special case of \cite[Corollary 7.1.4]{BerIHES},
and for $i_{\ast}$ it's a very special case of \cite[Theorem 3.7.2]{Hub96}.
Moreover, if $f:\mathcal{Y=\mathrm{Spec}\, B}\to\mathcal{X}$ is any
finite morphism with analytification $f^{\mathrm{an}}:Y=\mathrm{Spa}\, B\to X$,
then $\mu_{X}^{\ast}f_{\ast}\cong f_{\ast}^{\mathrm{an}}\mu_{Y}^{\ast}$
(by \cite[Theorem 3.7.2]{Hub96} and \cite[Proposition 2.6.4]{Hub96}
again) and $\mu_{Y}^{\ast}f^{\ast}\cong f^{\mathrm{an}\ast}\mu_{X}^{\ast}$
(by taking adjoints to the obvious equivalence $\mu_{X\ast}f_{\ast}^{\mathrm{an}}\cong f_{\ast}\mu_{Y\ast}$).
We'll use all of these compatibilities without further comment.
\begin{proof}[Proof of Theorem \ref{comparison}]
First, observe that all functors involved in the statement of the
theorem commute with filtered colimits: for $H_{\mathrm{\acute{e}t}}^{n}(\mathcal{X},-)$
this is standard, for $H_{\mathrm{\acute{e}t}}^{n}(X,-)$ this follows
from \cite[Lemma 2.3.13]{Hub96}, and for $\mu_{X}^{\ast}$ it is
trivial (because $\mu_{X}^{\ast}$ is a left adjoint). Writing $\mathscr{F}$
as the filtered colimit of its $m$-torsion subsheaves, we therefore
reduce to the case where $\mathscr{F}$ is killed by some integer
$m\geq1$. Since $\mathcal{X}$ is qcqs, we can write any sheaf of
$\mathbf{Z}/m\mathbf{Z}$-modules on $\mathcal{X}_{\mathrm{\acute{e}t}}$
as a filtered colimit of constructible sheaves of $\mathbf{Z}/m\mathbf{Z}$-modules,
cf. \cite[Tag 03SA]{Stacks}, which reduces us further to the case
where $\mathscr{F}$ is a constructible sheaf of $\mathbf{Z}/m\mathbf{Z}$-modules.

We now argue by induction on $\mathrm{dim}\,\mathcal{X}$ ($=\mathrm{dim}X$).
By Noether normalization for affinoids and the aforementioned compatibility
of $\mu^{\ast}$ with pushforward along finite maps, we can assume
that $A\simeq K\left\langle X_{1},\dots,K_{\mathrm{dim}X}\right\rangle $,
so in particular that $A$ is normal. Choose some dense Zariski-open
$j:\mathcal{U}\to\mathcal{X}$ such that $j^{\ast}\mathscr{F}$ is
locally constant. Writing $\mathscr{F}^{\mathrm{an}}=\mu_{X}^{\ast}\mathscr{F}$
for brevity and taking the cohomology of the sequence \[
0\to j_{!}j^{\ast}\mathscr{F}\to\mathscr{F}\to i_{\ast}i^{\ast}\mathscr{F}\to0\]
before and after applying $\mu^{\ast}$, we get a pair of long exact
sequence sitting in a commutative diagram\[
\xymatrix{\cdots\ar[r] & H_{\mathrm{\acute{e}t}}^{n}(\mathcal{X},j_{!}j^{\ast}\mathscr{F})\ar[r]\ar[d]^{\psi_{1,n}} & H_{\mathrm{\acute{e}t}}^{n}(\mathcal{X},\mathscr{F})\ar[r]\ar[d]^{\psi_{2,n}} & H_{\mathrm{\acute{e}t}}^{n}(\mathcal{Z},i^{\ast}\mathscr{F})\ar[r]\ar[d]^{\psi_{3,n}} & \cdots\\
\cdots\ar[r] & H_{\mathrm{\acute{e}t}}^{n}(X,j_{!}^{\mathrm{an}}j^{\mathrm{an}\ast}\mathscr{F}^{\mathrm{an}})\ar[r] & H_{\mathrm{\acute{e}t}}^{n}(X,\mathscr{F}^{\mathrm{an}})\ar[r] & H_{\mathrm{\acute{e}t}}^{n}(Z,i^{\mathrm{an}\ast}\mathscr{F}^{\mathrm{an}})\ar[r] & \cdots}
\]
(here we've freely used the various compatibilities of the four functors
and analytification). But then $\psi_{1,n}$ is an isomorphism for
all $n$ by Lemma \ref{keylemma}, and $\psi_{3,n}$ is an isomorphism
for all $n$ by the induction hypothesis, so $\psi_{2,n}$ is an isomorphism
for all $n$ by the five lemma.
\end{proof}
It remains to prove Lemma \ref{keylemma}. For this we make mild use
of cohomological descent; in particular, we need the following observation,
which might be useful in other contexts.
\begin{prop}
Let $f:Y\to X$ be a surjective qcqs map of rigid analytic spaces.
Then $f$ is universally of cohomological descent relative to the
étale topology.\end{prop}
\begin{proof}
Using Huber's general qcqs base change theorem \cite[Theorem 4.1.1.(c)']{Hub96}
and arguing as in the proof of \cite[Theorem 7.7]{Conrad}, one reduces
via \cite[Theorem 7.2]{Conrad} to the fact that if $Y$ is a qcqs
adic space which is topologically of finite type over a (possibly
higher rank) geometric point $\mathrm{Spa}(C,C^{+})$, then the structure
map $Y\to\mathrm{Spa}(C,C^{+})$ admits a section, which follows e.g.
as in \cite[Lemma 9.5]{Scholze}.
\end{proof}
In particular, if $\mathcal{X}=\mathrm{Spec}\, A$ is as above and
$f:\mathcal{Y}\to\mathcal{X}$ is a surjective finite morphism (or
more generally, a surjective proper morphism) with analytification
$f^{\mathrm{an}}:Y\to X=\mathrm{Spa}\, A$, then on the one hand $f$
is surjective and proper and therefore universally of cohomological
descent in the usual sense, while on the other hand $f^{\mathrm{an}}$
is surjective and qcqs and thus universally of cohomological descent
by the previous proposition. Considering the $0$-coskeleta $\mathcal{Y}_{/\mathcal{X}}^{\bullet}=\left(\mathcal{Y}_{/\mathcal{X}}^{(p)}=\underbrace{\mathcal{Y}\times_{\mathcal{X}}\cdots\times_{\mathcal{X}}\mathcal{Y}}_{p}\overset{\varepsilon_{p}}{\to}\mathcal{X}\right)_{p\geq0}$
and $Y_{/X}^{\bullet}=(\cdots)_{p\geq0}$ of the maps $f$ and $f^{\mathrm{an}}$,
we then get a pair of cohomological descent spectral sequences\[
E_{1}^{p,q}=H_{\mathrm{\acute{e}t}}^{q}(\mathcal{Y}_{/\mathcal{X}}^{(p)},\varepsilon_{p}^{\ast}\mathscr{F})\Rightarrow H_{\mathrm{\acute{e}t}}^{p+q}(\mathcal{X},\mathscr{F})\]
and\[
E_{1,\mathrm{an}}^{p,q}=H_{\mathrm{\acute{e}t}}^{q}(Y_{/X}^{(p)},\varepsilon_{p}^{\mathrm{an}\ast}\mathscr{F}^{\mathrm{an}})\Rightarrow H_{\mathrm{\acute{e}t}}^{p+q}(X,\mathscr{F}^{\mathrm{an}})\]
for any torsion abelian sheaf $\mathscr{F}$ on $\mathcal{X}_{\mathrm{\acute{e}t}}$,
as in \cite[Theorem 6.11]{Conrad}. By general nonsense, there is
a morphism from the first spectral sequence to the second compatible
with the evident maps between the individual terms in the $E_{1}$-pages
and in the abutments.
\begin{proof}[Proof of Lemma \ref{keylemma}]
We can assume that $\mathcal{U}$ is dense in $\mathcal{X}$. By
assumption, we can choose some surjective finite étale map $\pi:\mathcal{Y}\to\mathcal{U}$
such that $\pi^{\ast}\mathscr{F}$ is constant. Since $A$ is normal
and Nagata, we can apply the schemes version of Theorem \ref{extendingcovers}
to obtain a cartesian diagram\[
\xymatrix{\mathcal{Y}\ar[d]_{\pi}\ar[r]^{j_{0}} & \tilde{\mathcal{Y}}\ar[d]^{\tilde{\pi}}\\
\mathcal{U}\ar[r]^{j} & \mathcal{X}}
\]
where $\tilde{\pi}$ is finite and surjective and $\tilde{\mathcal{Y}}=\mathrm{Spec}\,\tilde{B}$
is normal (to be clear, $\tilde{\mathcal{Y}}$ is simply the normalization
of $\mathcal{X}$ in $\mathcal{Y}$, cf. Lemma \ref{extendingschemecovers}
below). Let $\tilde{\mathcal{Y}}_{/\mathcal{X}}^{\bullet}$ and $\mathcal{Y}_{/\mathcal{U}}^{\bullet}$
be the evident $0$-coskeleta, so we get a cartesian diagram\[
\xymatrix{\mathcal{Y}_{/\mathcal{U}}^{\bullet}\ar[d]_{\varepsilon_{\bullet}}\ar[r]^{j_{\bullet}} & \tilde{\mathcal{Y}}_{/\mathcal{X}}^{\bullet}\ar[d]^{\tilde{\varepsilon}_{\bullet}}\\
\mathcal{U}\ar[r]^{j} & \mathcal{X}}
\]
where the vertical maps are the evident augmentations. Denote the
analogous rigid analytic objects and diagrams by roman letters and
$(-)^{\mathrm{an}}$'s in the obvious way. Since $\tilde{\mathcal{Y}}\to\mathcal{X}$
and its analytification are universally of cohomological descent,
the discussion above gives a morphism of spectral sequences\begin{eqnarray*}
H_{\mathrm{\acute{e}t}}^{q}(\tilde{\mathcal{Y}}_{/\mathcal{X}}^{(p)},\tilde{\varepsilon}_{p}^{\ast}j_{!}\mathscr{F}) & \Rightarrow & H_{\mathrm{\acute{e}t}}^{p+q}(\mathcal{X},j_{!}\mathscr{F})\\
\downarrow\qquad &  & \qquad\downarrow\\
H_{\mathrm{\acute{e}t}}^{q}(\tilde{Y}_{/X}^{(p)},\tilde{\varepsilon}_{p}^{\mathrm{an}\ast}j_{!}^{\mathrm{an}}\mathscr{F}) & \Rightarrow & H_{\mathrm{\acute{e}t}}^{p+q}(X,j_{!}^{\mathrm{an}}\mathscr{F}^{\mathrm{an}})\end{eqnarray*}
and it now suffices to check that the individual maps on the $E_{1}$-page
are isomorphisms. Each of these maps sits as the lefthand vertical
arrow in a commutative squares\[
\xymatrix{H_{\mathrm{\acute{e}t}}^{q}(\tilde{\mathcal{Y}}_{/\mathcal{X}}^{(p)},\tilde{\varepsilon}_{p}^{\ast}j_{!}\mathscr{F})\ar[d]\ar[r] & H_{\mathrm{\acute{e}t}}^{q}(\tilde{\mathcal{Y}}_{/\mathcal{X}}^{(p)},j_{p!}\varepsilon_{p}^{\ast}\mathscr{F})\ar[d]\\
H_{\mathrm{\acute{e}t}}^{q}(\tilde{Y}_{/X}^{(p)},\tilde{\varepsilon}_{p}^{\mathrm{an}\ast}j_{!}^{\mathrm{an}}\mathscr{F})\ar[r] & H_{\mathrm{\acute{e}t}}^{q}(\tilde{Y}_{/X}^{(p)},j_{p!}^{\mathrm{an}}\varepsilon_{p}^{\mathrm{an}\ast}\mathscr{F})}
\]
where the horizontal arrows (exist and) are isomorphisms by proper
base change. It's now enough to show that the righthand vertical arrow
is an isomorphism; since each $\varepsilon_{p}^{\ast}\mathscr{F}$
is a constant sheaf, this reduces us to the special case of the lemma
where $\mathscr{F}$ is \emph{constant. }

Returning to the original notation of the lemma, we can now assume
that $\mathscr{F}$ is the constant\emph{ }sheaf associated with some
finite abelian group $G$. Again, we get a pair of long exact sequences
sitting in a commutative diagram\[
\xymatrix{\cdots\ar[r] & H_{\mathrm{\acute{e}t}}^{n}(\mathcal{X},j_{!}G)\ar[r]\ar[d]^{\psi_{1,n}} & H_{\mathrm{\acute{e}t}}^{n}(\mathcal{X},G)\ar[r]\ar[d]^{\psi_{2,n}} & H_{\mathrm{\acute{e}t}}^{n}(\mathcal{Z},G)\ar[r]\ar[d]^{\psi_{3,n}} & \cdots\\
\cdots\ar[r] & H_{\mathrm{\acute{e}t}}^{n}(X,j_{!}^{\mathrm{an}}G)\ar[r] & H_{\mathrm{\acute{e}t}}^{n}(X,G)\ar[r] & H_{\mathrm{\acute{e}t}}^{n}(Z,G)\ar[r] & \cdots}
\]
as before. But now $\psi_{2,n}$ and $\psi_{3,n}$ are isomorphisms
for all $n$ by \cite[Corollary 3.2.3]{Hub96}, so applying the five
lemma again we conclude that $\psi_{1,n}$ is an isomorphism for all
$n$, as desired.
\end{proof}
In this argument, we used the following schemes version of Theorem
\ref{extendingcovers}.
\begin{lem}
\label{extendingschemecovers}Let $X$ be any scheme which is normal
and Nagata, and let $Z\subset X$ be any closed nowhere-dense subset.
Then the restriction functor\begin{eqnarray*}
\left\{ \begin{array}{c}
\mathrm{covers\, of\,}X\\
\mathrm{\acute{e}tale\, over\,}X\smallsetminus Z\end{array}\right\}  & \to & \left\{ \begin{array}{c}
\mathrm{finite\,\acute{e}tale\, covers}\\
\mathrm{of\,}X\smallsetminus Z\end{array}\right\} \\
Y & \mapsto & Y\times_{X}(X\smallsetminus Z)\end{eqnarray*}
is an equivalence of categories, with essential inverse given by sending
a finite étale cover $U\to X\smallsetminus Z$ to the normalization
of $X$ in $U$.
\end{lem}
We remind the reader that a scheme is Nagata if it is locally Noetherian
and admits an open covering by spectra of universally Japanese rings,
cf. \cite[Tag 033R]{Stacks}.
\begin{proof}[Proof sketch]
Let $U\to X\smallsetminus Z$ be a finite étale cover, and let $\tilde{U}\to X$
be the normalization of $X$ in $U$ as in \cite[Tag 035G]{Stacks}.
By \cite[Tag 03GR]{Stacks}, $\tilde{U}\to X$ is a finite morphism.
By \cite[Tags 035K and 03GP]{Stacks}, the diagram\[
\xymatrix{U\ar[r]\ar[d] & \tilde{U}\ar[d]\\
X\smallsetminus Z\ar[r] & X}
\]
is cartesian. By \cite[Tag 035L]{Stacks}, the scheme $\tilde{U}$
is normal. The remaining verifications are left as an exercise for
the interested reader.
\end{proof}

\subsection{The reduction step}

In this section we deduce Theorem \ref{avfull} from Theorem \ref{avconstant}.
For clarity we focus on the vanishing statement in the theorem; it's
easy to see that the following argument also reduces the finiteness
of the groups $H_{\mathrm{\acute{e}t}}^{i}(X_{\widehat{\overline{X}}},\mathscr{F})$
to finiteness in the special case where $\mathscr{F}=\mathbf{Z}/n\mathbf{Z}$
is constant, and finiteness in the latter case follows from \cite[Theorem 1.1.1]{BerFiniteness}.
\begin{proof}[Proof of Theorem \ref{avfull}]
Fix a nonarchimedean field $K$ and a coefficient ring $\Lambda=\mathbf{Z}/n\mathbf{Z}$
as in the theorem. In what follows, {}``sheaf'' is shorthand for
{}``étale sheaf of $\Lambda$-modules''. For nonnegative integers
$d,i$, consider the following statement. 

Statement $\mathcal{T}_{d,i}$: {}``For all $K$-affinoids $X$ of
dimension $\leq d$, all Zariski-constructible sheaves $\mathscr{F}$
on $X$, and all integers $j>i$, we have $H_{\mathrm{\acute{e}t}}^{j}(X_{\widehat{\overline{K}}},\mathscr{F})=0$.'' 

We are trying to prove that $\mathcal{T}_{d,d}$ is true for all $d\geq0$.
The idea is to argue by ascending induction on $d$ and descending
induction on $i$. More precisely, it clearly suffices to assume the
truth of $\mathcal{T}_{d-1,d-1}$ and then show that $\mathcal{T}_{d,i+1}$
implies $\mathcal{T}_{d,i}$ for any $i\geq d$; as noted in the introduction,
$\mathcal{T}_{d,2d}$ is true for any $d\geq0$, which gives a starting
place for the descending induction. 

We break the details into several steps.

\textbf{Step One. }\emph{Suppose that $\mathcal{T}_{d-1,d-1}$ holds.
Then for any $d$-dimensional affinoid $X$, any Zariski-constructible
sheaf $\mathscr{F}$ on $X$, and any dense Zariski-open subset $j:U\to X$,
the natural map $H_{\mathrm{\acute{e}t}}^{i}(X_{\widehat{\overline{K}}},j_{!}j^{\ast}\mathscr{F})\to H_{\mathrm{\acute{e}t}}^{i}(X_{\widehat{\overline{K}}},\mathscr{F})$
is surjective for $i=d$ and bijective for $i>d$.}

Letting $i:Z\to X$ denote the closed complement, this is immediate
by looking at the long exact sequence\[
\cdots\to H_{\mathrm{\acute{e}t}}^{i-1}(Z_{\widehat{\overline{K}}},i^{\ast}\mathscr{F})\to H_{\mathrm{\acute{e}t}}^{i}(X_{\widehat{\overline{K}}},j_{!}j^{\ast}\mathscr{F})\to H_{\mathrm{\acute{e}t}}^{i}(X_{\widehat{\overline{K}}},\mathscr{F})\to H_{\mathrm{\acute{e}t}}^{i}(Z_{\widehat{\overline{K}}},i^{\ast}\mathscr{F})\to\cdots\]
associated with the short exact sequence\[
0\to j_{!}j^{\ast}\mathscr{F}\to\mathscr{F}\to i_{\ast}i^{\ast}\mathscr{F}\to0\]
and then applying $\mathcal{T}_{d-1,d-1}$ to control the outer terms.

\textbf{Step Two. }\emph{For any $d,i$, $\mathcal{T}_{d,i}$ holds
if and only if it holds for all }normal\emph{ affinoids.}

One direction is trivial. For the other direction, note that by Noether
normalization for affinoids \cite[Corollary 6.1.2/2]{BGR}, any $d$-dimensional
affinoid $X$ admits a finite map \[
\tau:X\to\mathbf{B}^{d}=\mathrm{Spa}\, K\left\langle T_{1},\dots,T_{d}\right\rangle ,\]
and $\tau_{\ast}=R\tau_{\ast}$ preserves Zariski-constructibility
by Proposition \ref{finitepush}.

\textbf{Step Three.}%
\footnote{This step was inspired by some constructions in Nori's beautiful paper
\cite{MR1940678}.%
}\textbf{ }\emph{Suppose that $\mathcal{T}_{d-1,d-1}$ holds. Then
for any $d$-dimensional normal affinoid $X$, any dense Zariski-open
subset $j:U\to X$, and any locally constant constructible sheaf $\mathscr{H}$
on $U$, we can find a Zariski-constrictible sheaf $\mathscr{G}$
on $X$ together with a surjection $s:\mathscr{G}\to j_{!}\mathscr{H}$,
such that moreover $H_{\mathrm{\acute{e}t}}^{i}(X_{\widehat{\overline{K}}},\mathscr{G})=0$
for all $i>d$.}

To prove this, suppose we are given $X,U$, and $\mathscr{H}$ as
in the statement. By definition, we can find a finite étale cover
$\pi:Y\to U$ such that $\pi^{\ast}\mathscr{H}$ is constant, i.e.
such that there exists a surjection $\Lambda_{Y}^{n}\to\pi^{\ast}\mathscr{H}$
for some $n$; fix such a surjection. This is adjoint to a surjection
$\pi_{!}(\Lambda_{Y}^{n})\to\mathscr{H}$,%
\footnote{Surjectivity here can be checked either by a direct calculation or
by {}``pure thought'' ($\pi_{!}$ is left adjoint to $\pi^{\ast}$,
and left adjoints preserve epimorphisms). %
} which extends by zero to a surjection $s:j_{!}\pi_{!}(\Lambda_{Y}^{n})\to j_{!}\mathscr{H}$.
We claim that the sheaf $\mathscr{G}=j_{!}\pi_{!}(\Lambda_{Y}^{n})$
has the required properties. Zariski-constructibility is clear from
the identification $\pi_{!}=\pi_{\ast}$ and Proposition \ref{finitepush}.
For the vanishing statement, we apply Theorem \ref{extendingcovers}
to extend $Y\to U$ to a cover $\tilde{\pi}:\tilde{Y}\to X$ sitting
in a cartesian diagram\[
\xymatrix{Y\ar[d]_{\pi}\ar[r]^{h} & \tilde{Y}\ar[d]^{\tilde{\pi}}\\
U\ar[r]^{j} & X}
\]
where $\tilde{Y}$ is a normal affinoid and $h$ is a dense Zariski-open
immersion. By proper base change and the finiteness of $\tilde{\pi}$,
we get isomorphisms\[
\mathscr{G}=j_{!}\pi_{!}(\Lambda_{Y}^{n})\cong\tilde{\pi}_{!}h_{!}(\Lambda_{Y}^{n})\cong\tilde{\pi}_{\ast}h_{!}(\Lambda_{Y}^{n}),\]
so $H_{\mathrm{\acute{e}t}}^{i}(X_{\widehat{\overline{K}}},\mathscr{G})\cong H_{\mathrm{\acute{e}t}}^{i}(\tilde{Y}_{\widehat{\overline{K}}},h_{!}\Lambda_{Y})^{\oplus n}$.
$ $Now, writing $i:V\to\tilde{Y}$ for the closed complement of $Y$,
we get exact sequences\[
H_{\mathrm{\acute{e}t}}^{i-1}(V_{\widehat{\overline{K}}},\Lambda_{V})\to H_{\mathrm{\acute{e}t}}^{i}(\tilde{Y}_{\widehat{\overline{K}}},h_{!}\Lambda_{Y})\to H_{\mathrm{\acute{e}t}}^{i}(\tilde{Y}_{\widehat{\overline{K}}},\Lambda_{\tilde{Y}})\]
for all $i$. Examining this sequence for any fixed $i>d$, we see
that the rightmost term vanishes by Theorem \ref{avconstant}, while
the leftmost term vanishes by the assumption that $\mathcal{T}_{d-1,d-1}$
holds.%
\footnote{One really needs the induction hypothesis to control the leftmost
term here, since $V$ may not be normal.%
} Therefore $H_{\mathrm{\acute{e}t}}^{i}(\tilde{Y}_{\widehat{\overline{K}}},h_{!}\Lambda_{Y})=0$
for $i>d$, as desired.

\textbf{Step Four. }\emph{Suppose that $\mathcal{T}_{d-1,d-1}$ holds.
Then $\mathcal{T}_{d,i+1}$ implies $\mathcal{T}_{d,i}$ for any $i\geq d$.}

Fix $d$ and $i\geq d$ as in the statement, and assume $\mathcal{T}_{d,i+1}$
is true. Let $X$ be a $d$-dimensional affinoid, and let $\mathscr{F}$
be a Zariski-constructible sheaf on $X$. We need to show that $H_{\mathrm{\acute{e}t}}^{i+1}(X_{\widehat{\overline{K}}},\mathscr{F})=0$.
By Step Two, we can assume $X$ is normal (or even that $X$ is the
$d$-dimensional affinoid ball). By Step One, it suffices to show
that $H_{\mathrm{\acute{e}t}}^{i+1}(X_{\widehat{\overline{K}}},j_{!}j^{\ast}\mathscr{F})=0$
where $j:U\to X$ is the inclusion of any dense Zariski-open subset.
Fix a choice of such a $U$ with the property that $j^{\ast}\mathscr{F}$
is locally constant. By Step Three, we can choose a Zariski-constructible
sheaf $\mathscr{G}$ on $X$ and a surjection $s:\mathscr{G}\to j_{!}j^{\ast}\mathscr{F}$
such that $H_{\mathrm{\acute{e}t}}^{n}(X_{\widehat{\overline{K}}},\mathscr{G})=0$
for all $n>d$. By Proposition \ref{twooutofthree}, the sheaf $\mathscr{K}=\ker s$
is Zariski-constructible. Now, looking at the exact sequence\[
H_{\mathrm{\acute{e}t}}^{i+1}(X_{\widehat{\overline{K}}},\mathscr{G})\to H_{\mathrm{\acute{e}t}}^{i+1}(X_{\widehat{\overline{K}}},j_{!}j^{\ast}\mathscr{F})\to H_{\mathrm{\acute{e}t}}^{i+2}(X_{\widehat{\overline{K}}},\mathscr{K}),\]
we see that the leftmost term vanishes by the construction of $\mathscr{G}$,
while the rightmost term vanishes by the induction hypothesis. Therefore\[
H_{\mathrm{\acute{e}t}}^{i+1}(X_{\widehat{\overline{K}}},j_{!}j^{\ast}\mathscr{F})=0,\]
as desired.
\end{proof}

\subsection{Constant coefficients}

In this section we prove Theorem \ref{avconstant}. The following
technical lemma plays an important role in the argument.
\begin{lem}
\label{affinoidpbexcellent}Let $K$ be a complete discretely valued
nonarchimedean field, and let $A$ be a reduced $K$-affinoid algebra.
Then $A^{\circ}$ is an excellent Noetherian ring. Moreover, the strict
Henselization of any localization of $A^{\circ}$ is excellent as
well.\end{lem}
\begin{proof}
By Noether normalization for affinoids and \cite[Corollary 6.4.1/6]{BGR},
$A^{\circ}$ can be realized as a module-finite integral extension
of $\mathcal{O}_{K}\left\langle T_{1},\dots,T_{n}\right\rangle $
with $n=\mathrm{dim}\, A$. By a result of Valabrega (cf. \cite[Proposition 7]{ValExcellent75}
and \cite[Theorem 9]{ValExcellent76}), the convergent power series
ring $\mathcal{O}_{K}\left\langle T_{1},\dots,T_{n}\right\rangle $
is excellent for any complete discrete valuation ring $\mathcal{O}_{K}$.
Since excellence propagates along finite type ring maps and localizations,
cf. \cite[Tag 07QU]{Stacks}, we see that $A^{\circ}$ and any localization
thereof is excellent. Now, by a result of Greco \cite[Corollary 5.6.iii]{Greco},
the strict Henselization of any excellent local ring is excellent,
which gives what we want.
\end{proof}
We also need the following extremely powerful theorem of Gabber.
\begin{thm}[Gabber]
\label{gabbervanishing}Let $B$ be a quasi-excellent strictly Henselian
local ring, and let $U\subset\mathrm{Spec}B$ be an affine open subscheme.
Then $H_{\mathrm{\acute{e}t}}^{i}(U,\mathbf{Z}/n\mathbf{Z})=0$ for
any $i>\mathrm{dim}B$ and any integer $n$ invertible in $B$.\end{thm}
\begin{proof}
This is a special case of Gabber's affine Lefschetz theorem for quasi-excellent
schemes, cf. Corollaire XV.1.2.4 in \cite{GabberLefschetzPS}.
\end{proof}
Finally, we recall the following strong form of Artin's vanishing
theorem \cite[\S XIV.3]{SGA4v3}.
\begin{thm}[Artin]
Let $X$ be an affine variety over a separably closed field $k$,
and let $\mathscr{F}$ be a torsion abelian sheaf on $X_{\mathrm{\acute{e}t}}$.
Set \[
\delta(\mathscr{F})=\mathrm{sup}\left\{ \mathrm{tr.deg}\, k(x)/k\mid\mathscr{F}_{\overline{x}}\neq0\right\} .\]
Then $H_{\mathrm{\acute{e}t}}^{i}(X,\mathscr{F})=0$ for all $i>\delta(\mathscr{F})$.$\phantom{}$\end{thm}
\begin{proof}[Proof of Theorem \ref{avconstant}]
Let $X=\mathrm{Spa}\, A$ be a $K$-affinoid as in the theorem. After
replacing $K$ by $\widehat{K^{\mathrm{nr}}}$ and $X$ by $X_{\widehat{K^{\mathrm{nr}}}}^{\mathrm{red}}$,
we can assume that $A$ is reduced and that $K$ has separably closed
residue field $k$. By an easy induction we can also assume that $n=l$
is prime. For notational simplicity we give the remainder of the proof
in the case where $\mathrm{char}(k)=p>0$; the equal characteristic
zero case is only easier.

By e.g. Corollary 2.4.6 in \cite{BerIHES}, $\mathrm{Gal}_{\overline{K}/K}$
sits in a short exact sequence\[
1\to P\to\mathrm{Gal}_{\overline{K}/K}\to T\simeq\prod_{q\neq p}\mathbf{Z}_{q}\to1\]
where $P$ is pro-$p$. In particular, if $L\subset\overline{K}$
is any finite extension of $K$, then \[
H^{i}(\mathrm{Gal}_{\overline{K}/L},\mathbf{Z}/l\mathbf{Z})\simeq\begin{cases}
\mathbf{Z}/l\mathbf{Z}\,\mathrm{if}\, i=0,1\\
0\,\mathrm{if}\, i>1\end{cases}.\]
For any such $L$, look at the Cartan-Leray spectral sequence\[
E_{2}^{i,j}=H^{i}(\mathrm{Gal}_{\overline{K}/L},H_{\mathrm{\acute{e}t}}^{j}(X_{\widehat{\overline{K}}},\mathbf{Z}/l\mathbf{Z}))\Rightarrow H_{\mathrm{\acute{e}t}}^{i+j}(X_{L},\mathbf{Z}/l\mathbf{Z}).\]
The group $\oplus_{j\geq0}H_{\mathrm{\acute{e}t}}^{j}(X_{\widehat{\overline{K}}},\mathbf{Z}/l\mathbf{Z})$
is finite by vanishing in degrees $>2\,\mathrm{dim}X$ together with
\cite[Theorem 1.1.1]{BerFiniteness}, so all the Galois actions in
the $E_{2}$-page are trivial for any large enough $L$,%
\footnote{This trick was inspired by a discussion of Poincaré dualities in an
IHES lecture by Peter Scholze, cf. https://www.youtube.com/watch?v=E3zAEqkd9cQ.%
} in which case we can rewrite the spectral sequence as \[
E_{2}^{i,j}=H^{i}(\mathrm{Gal}_{\overline{K}/L},\mathbf{Z}/l\mathbf{Z})\otimes H_{\mathrm{\acute{e}t}}^{j}(X_{\widehat{\overline{K}}},\mathbf{Z}/l\mathbf{Z})\Rightarrow H_{\mathrm{\acute{e}t}}^{i+j}(X_{L},\mathbf{Z}/l\mathbf{Z}).\]
Now if $d$ is the largest integer such that $H_{\mathrm{\acute{e}t}}^{d}(X_{\widehat{\overline{K}}},\mathbf{Z}/l\mathbf{Z})\neq0$,
then the $E_{2}^{1,d}$ term survives the spectral sequence, so $H_{\mathrm{\acute{e}t}}^{d+1}(X_{L},\mathbf{Z}/l\mathbf{Z})\neq0$;
moreover, this latter group coincides with $H_{\mathrm{\acute{e}t}}^{d+1}(X_{L^{\mathrm{t}}},\mathbf{Z}/l\mathbf{Z})$
where $L^{\mathrm{t}}$ denotes the maximal subfield of $L$ tamely
ramified over $K$ (this is immediate from Cartan-Leray, since $\mathrm{Gal}_{L/L^{\mathrm{t}}}$
is a $p$-group). By Proposition 2.4.7 in \cite{BerIHES} the residue
field of $L^{\mathrm{t}}$ is still separably closed. It thus suffices
to prove the following statement:

($\dagger$) For any reduced affinoid $X=\mathrm{Spa}\, A$ over a
complete discretely valued nonarchimedean field $K$ with separably
closed residue field $k$ of characteristic $p>0$, we have $H_{\mathrm{\acute{e}t}}^{i}(X,\mathbf{Z}/l\mathbf{Z})=0$
for all $i>1+\mathrm{dim}X$ and all primes $l\neq p$.

Fix a uniformizer $\varpi\in\mathcal{O}_{K}$. Set $\mathcal{X}=\mathrm{Spec}\, A^{\circ}$
and $\mathcal{X}_{s}=\mathrm{Spec}\, A^{\circ}/\varpi$, so $\mathcal{X}_{s}$
is an affine variety over $k$. As in \cite[\S3.5]{Hub96} or \cite{BerVanishing},
there is a natural map of sites $\lambda:X_{\mathrm{\acute{e}t}}\to\mathcal{X}_{s,\mathrm{\acute{e}t}}$,
corresponding to the natural functor \begin{eqnarray*}
\mathcal{X}_{s,\mathrm{\acute{e}t}} & \to & X_{\mathrm{\acute{e}t}}\\
\mathcal{U}/\mathcal{X}_{s} & \mapsto & \eta(\mathcal{U})/X\end{eqnarray*}
given by (uniquely) deforming an étale map $\mathcal{U}\to\mathcal{X}_{s}$
to a $\varpi$-adic formal scheme étale over $\mathrm{Spf}\, A^{\circ}$
and then passing to rigid generic fibers. (We follow Huber's notation
in writing $\lambda$ - Berkovich denotes this map by $\Theta$.)
For any abelian étale sheaf $\mathscr{F}$ on $X$, derived pushforward
along $\lambda$ gives rise to the so-called nearby cycle sheaves
$R^{j}\lambda_{\ast}\mathscr{F}$ on $\mathcal{X}_{s,\mathrm{\acute{e}t}}$,
which can be calculated as the sheafifications of the presheaves $\mathcal{U}\mapsto H_{\mathrm{\acute{e}t}}^{j}(\eta(\mathcal{U}),\mathscr{F})$,
and there is a spectral sequence\[
H_{\mathrm{\acute{e}t}}^{i}(\mathcal{X}_{s},R^{j}\lambda_{\ast}\mathscr{F})\Rightarrow H_{\mathrm{\acute{e}t}}^{i+j}(X,\mathscr{F}),\]
cf. Proposition 4.1 and Corollary 4.2.(iii) in \cite{BerVanishing}.
Taking $\mathscr{F}=\mathbf{Z}/l\mathbf{Z}$, we see that to prove
$(\dagger)$ it's enough to show that $H_{\mathrm{\acute{e}t}}^{i}(\mathcal{X}_{s},R^{j}\lambda_{\ast}\mathbf{Z}/l\mathbf{Z})=0$
for any $j\geq0$ and any $i>1+\mathrm{dim}X-j$. By the strong form
of the Artin vanishing theorem recalled above, we're reduced to proving
that if $x\in\mathcal{X}_{s}$ is any point such that \[
(R^{j}\lambda_{\ast}\mathbf{Z}/l\mathbf{Z})_{\overline{x}}\neq0,\]
then $\mathrm{tr.deg}\, k(x)/k\leq1+\mathrm{dim}X-j$.

We check this by a direct computation. So, let $x\in\mathcal{X}_{s}$
be any point, and let $\mathfrak{p}_{x}\subset A^{\circ}$ be the
associated prime ideal. Crucially, we have a {}``purely algebraic''
description of the stalk $(R^{j}\lambda_{\ast}\mathbf{Z}/l\mathbf{Z})_{\overline{x}}$:
letting $\mathcal{O}_{\mathcal{X},\overline{x}}$ denote the strict
Henselization of $\mathcal{O}_{\mathcal{X},x}=(A^{\circ})_{\mathfrak{p}_{x}}$
as usual, then\[
(R^{j}\lambda_{\ast}\mathbf{Z}/l\mathbf{Z})_{\overline{x}}\cong H_{\mathrm{\acute{e}t}}^{j}\left(\mathrm{Spec}\,\mathcal{O}_{\mathcal{X},\overline{x}}[\tfrac{1}{\varpi}],\mathbf{Z}/l\mathbf{Z}\right).\;\;\;\;\;\;\;\;\;\;\;\;(\ast)\]
This is a special case of \cite[Theorem 3.5.10]{Hub96}, and it's
remarkable that we have a description like this which doesn't involve
taking some completion. By Lemma \ref{affinoidpbexcellent}, $\mathcal{O}_{\mathcal{X},\overline{x}}$
is excellent, so Theorem \ref{gabbervanishing} implies that $H_{\mathrm{\acute{e}t}}^{j}(U,\mathbf{Z}/l\mathbf{Z})=0$
for any\emph{ }open affine subscheme $U\subset\mathrm{Spec}\,\mathcal{O}_{\mathcal{X},\overline{x}}$
and any $j>\mathrm{dim}\,\mathcal{O}_{\mathcal{X},\overline{x}}$.
In particular, taking $U=\mathrm{Spec}\,\mathcal{O}_{\mathcal{X},\overline{x}}[\tfrac{1}{\varpi}]$
and applying Huber's formula ($\ast$) above, we see that if $j$
is an integer such that $(R^{j}\lambda_{\ast}\mathbf{Z}/l\mathbf{Z})_{\overline{x}}\neq0$,
then necessarily \[
j\leq\mathrm{dim}\,\mathcal{O}_{\mathcal{X},\overline{x}}=\mathrm{dim}\,\mathcal{O}_{\mathcal{X},x}=\mathrm{ht}\,\mathfrak{p}_{x},\]
where the first equality follows from e.g. \cite[Tag 06LK]{Stacks}.
Writing $R=A^{\circ}/\varpi$ and $\overline{\mathfrak{p}_{x}}=\mathfrak{p}_{x}/\varpi\subset R$,
we then have\begin{eqnarray*}
j+\mathrm{tr.deg}\, k(x)/k & \leq & \mathrm{ht}\,\mathfrak{p}_{x}+\mathrm{tr.deg}\, k(x)/k\\
 & = & 1+\mathrm{ht}\,\overline{\mathfrak{p}_{x}}+\mathrm{tr}.\mathrm{deg}\, k(x)/k\\
 & = & 1+\mathrm{dim}R_{\overline{\mathfrak{p}_{x}}}+\mathrm{dim}R/\overline{\mathfrak{p}_{x}}\\
 & \leq & 1+\mathrm{dim}R\\
 & \leq & 1+\mathrm{dim}X.\end{eqnarray*}
Here the second line follows from the fact that $\varpi\in\mathfrak{p}_{x}$
is part of a system of parameters of $(A^{\circ})_{\mathfrak{p}_{x}}$,
so $\mathrm{ht}\,\overline{\mathfrak{p}_{x}}=\mathrm{dim}R_{\overline{\mathfrak{p}_{x}}}=\mathrm{dim}(A^{\circ})_{\mathfrak{p}_{x}}-1$;
the third line is immediate from the equality $\mathrm{tr}.\mathrm{deg}\, k(x)/k=\mathrm{dim}R/\overline{\mathfrak{p}_{x}}$,
which is a standard fact about domains of finite type over a field;
the fourth line is trivial; and the fifth line follows from the fact
that $R$ is module-finite over $k[T_{1},\dots,T_{n}]$ with $n=\mathrm{dim}X$.
But then\[
\mathrm{tr.deg}\, k(x)/k\leq1+\mathrm{dim}X-j,\]
as desired.
\end{proof}

\subsection{Stein spaces}

We end with the following slight generalization of the main theorem.
\begin{cor}
\label{steincase}Let $X$ be a rigid space over a characteristic
zero complete discretely valued nonarchimedean field $K$ which is
\emph{weakly Stein }in the sense that it admts an admissible covering
$X=\cup_{n\geq1}U_{n}$ by a nested sequence of open affinoid subsets
$U_{1}\subset U_{2}\subset U_{3}\subset\cdots$. Let $\mathscr{F}$
be any Zariski-constructible sheaf of $\mathbf{Z}/n\mathbf{Z}$-modules
on $X_{\mathrm{\acute{e}t}}$ for some $n$ prime to the residue characteristic
of $K$. Then\[
H_{\mathrm{\acute{e}t}}^{i}(X_{\widehat{\overline{K}}},\mathscr{F})=0\]
for all $i>\mathrm{dim}X$.\end{cor}
\begin{proof}
By \cite[Lemma 3.9.2]{Hub96}, we have a short exact sequence\[
0\to\underset{\leftarrow n}{\mathrm{lim}^{1}}H_{\mathrm{\acute{e}t}}^{i-1}(U_{n,\widehat{\overline{K}}},\mathscr{F})\to H_{\mathrm{\acute{e}t}}^{i}(X_{\widehat{\overline{K}}},\mathscr{F})\to\underset{\leftarrow n}{\mathrm{lim}}\, H_{\mathrm{\acute{e}t}}^{i}(U_{n,\widehat{\overline{K}}},\mathscr{F})\to0.\]
But the groups $H_{\mathrm{\acute{e}t}}^{i}(U_{n,\widehat{\overline{K}}},\mathscr{F})$
are finite, so the $\mathrm{lim}^{1}$ term vanishes, and the result
now follows from Theorem \ref{avfull}.
\end{proof}
This argument also shows that Conjecture 1.2, for a fixed choice of
$C$, is equivalent to the apparently more general conjecture that
the cohomology of any Zariski-constructible sheaf on any weakly Stein
space $X$ over $C$ vanishes in all degrees $>\mathrm{dim}X$.

\providecommand{\bysame}{\leavevmode\hbox to3em{\hrulefill}\thinspace}
\providecommand{\MR}{\relax\ifhmode\unskip\space\fi MR }
\providecommand{\MRhref}[2]{%
  \href{http://www.ams.org/mathscinet-getitem?mr=#1}{#2}
}
\providecommand{\href}[2]{#2}

\end{document}